\documentclass[12 pt]{amsart}
\usepackage{amscd,amsfonts,amssymb,amsmath}
\usepackage{hyperref}
\usepackage{epsfig}
\usepackage{mathtools}
\usepackage[T5]{fontenc}

\newtheorem{theorem}{Theorem}[section]
\newtheorem{corollary}[theorem]{Corollary}
\newtheorem{lemma}[theorem]{Lemma}
\newtheorem{proposition}[theorem]{Proposition}

\theoremstyle{definition}
\newtheorem{definition}[theorem]{Definition}
\newtheorem{conjecture}[theorem]{Conjecture}

\newtheorem{problem}[theorem]{Problem}

\theoremstyle{remark}
\newtheorem{remark}[theorem]{Remark}
\newtheorem*{ack}{Acknowledgments}

\numberwithin{equation}{section}
\usepackage[all,cmtip]{xy}
\usepackage{xcolor}

\usepackage[shortlabels]{enumitem}

\newcommand{\Alex}{\operatorname{Alex}}
\newcommand{\Aut}{\operatorname{Aut}}

\newcommand{\Conj}{\operatorname{Conj}}
\newcommand{\Core}{\operatorname{Core}}

\newcommand{\cl}{\operatorname{cl}}
\newcommand{\Cs}{\operatorname{Cs}}
\newcommand{\cw}{\operatorname{cw}}

\newcommand{\T}{\operatorname{T}}

\newcommand{\R}{\operatorname{R}}

\newcommand{\GL}{\operatorname{GL}}
\newcommand{\Aff}{\operatorname{Aff}}
\newcommand{\Nuc}{\operatorname{Nuc}}
\newcommand{\SC}{\operatorname{SC}}
\newcommand{\Z}{\mathbb{Z}}

\newcommand{\id}{\mathrm{id}}

\newcommand{\bij}{\xrightarrow{\sim}}
\renewcommand{\epsilon}{\varepsilon}
\newcommand{\inv}{^{-1}}
\renewcommand{\k}{\Bbbk}
\newcommand{\F}{\mathbb F}
\newcommand{\chara}{\operatorname{char}}
\begin{document}
\title[Idempotents, automorphisms, and commutators of quandle algebras]{Idempotents, automorphism groups, and commutator widths of quandle algebras}
\author{Birama Sangare}
\author{L\d\uhorn c Ta}

\date{\today}

\address{Novosibirsk State  University, 2 Pirogova Street, 630090, Novosibirsk, Russia.}
\email{b.sangare@g.nsu.ru}

\address{University of Pittsburgh, Pittsburgh, Pennsylvania 15260, United States.}
\email{ldt37@pitt.edu}


\begin{abstract} The paper develops the theory of quandle algebras. We show that over integral domains of characteristic other than 2, quandle algebras of ordered commutative quandles have no nontrivial idempotents. We also compute the automorphism groups of quandle algebras of trivial quandles and dihedral quandles of odd orders over arbitrary commutative rings, with partial results for dihedral quandles of even orders. Finally, a computer search provides the first examples of quandle algebras of commutator width 2.
\end{abstract}
\maketitle


\section{Introduction}
Quandles are algebraic structures with a nonassociative binary operation that satisfies axioms encoding the three Reidemeister moves of diagrams of links in three-space. Although quandles originated from knot theory \cite{Joyce, Matveev}, they also enjoy applications in quantum algebra \cite{Andruskiewitsch, BES, Eisermann}, algebraic geometry \cite{LT2}, and the theory of Riemannian symmetric spaces \cite{Loos}, among other areas of mathematics. In 1982, Joyce \cite{Joyce} and Matveev \cite{Matveev} independently introduced quandles and used them to construct complete invariants of non-split links up to orientation reversal. Of course, it is generally difficult to determine whether two quandles are isomorphic. This motivated a search for new properties and invariants of quandles themselves. We refer the reader to the articles \cite{Carter, Kamada, Nelson} for introductions to the theory and its historical development.
\par

Automorphisms of quandles, which reveal a lot about their internal structures, have been investigated in great detail in a series of papers \cite{BDS, BarTimSin, Elhamdadi2012}. Related work \cite{BSS1, BSS2} have used ideas from combinatorial group theory to show that free quandles and link quandles are residually finite.
\par

In an attempt to linearize the study of quandles, a theory of quandle rings analogous to the classical theory of group rings was proposed in \cite{BPS}, where several interconnections between quandles and their associated quandle rings were investigated, and an analogue of the group ring isomorphism problem for quandle rings was proposed. The work was developed further in \cite{EFT}, where examples of nonisomorphic finite quandles with isomorphic quandle rings have been given.
\par

The purpose of this paper is to develop the theory of quandle algebras of commutative quandles. Following \cite{BPS}, given a quandle (resp.\ rack) $Q$ and a unital associative commutative ring $\Bbbk$, the quandle (resp.\ rack) algebra $\Bbbk[Q]$ is defined as the free $\k$-module generated by $Q$ equipped with the nonassociative multiplication inherited from $Q$ (see  Section \ref{sec-prelim}).

As an analogue of Kaplansky's unit conjecture in the study of group algebras, the authors of \cite{BPS-1} proposed the study of idempotents in quandle algebras. They computed idempotents and $\Bbbk$-algebra automorphisms of quandle algebras for certain quandles of small orders, including all quandles of order $3$. In particular, they computed idempotents in the quandle algebras $\Bbbk[\T]$, $\mathbb{Z}[\R_3]$, $\mathbb{Z}[\R_4]$, and $\mathbb{Z}[\operatorname{J}_{3}]$, where $\T$ is any trivial quandle, $\R_n$ is a dihedral quandle of order $n$, and $\operatorname{J}_{3}$ is the 3-element quandle of Joyce \cite{Joyce-Thesis}. 
 
The present paper is organized as follows. In Section \ref{sec-prelim}, we recall some basic definitions and examples from the theory of quandles and quandle algebras. 
 
In Section \ref{Sec 4}, we investigate idempotents in quandle rings of commutative quandles. In \cite[Conj.\ 13.19]{BES}, it was conjectured that if $Q$ is a finite Latin quandle, then the only idempotents of $\mathbb{Z}[Q]$ are the trivial idempotents. We verify this conjecture for all commutative ordered quandles; in fact, the statement holds over arbitrary integral domains $\k$ with $\chara \k\neq 2$ (Theorem \ref{Theorem 4.8}). Over such $\k$ with $Q$ a commutative quandle, we also show that $\k[Q]$ contains no nontrivial idempotents of length $2$, and we completely characterize nontrivial idempotents of length $3$ (Theorem \ref{Theorem 4.6}). This addresses \cite[Conj.\ 3.10]{MBMD}.

In Section \ref{sec-auto}, we compute the automorphism groups of the quandle algebras of all finite trivial quandles and dihedral quandles of odd orders. More precisely, we prove that $\Aut(\k[\T_{n}])\cong\Aff_{n-1}(\k)$, where $\T_n$ is
the trivial quandle of order $n$ (Theorem \ref{thm:aff}). Also, we prove that $\Aut(\k[\R_{2n}])\cong \Aff_1(\SC_n(\k))\rtimes \Aut(\k[\R_n])$ for all odd integers $n\geq 1$ (Theorem \ref{thm:r2n}), where $\SC_n(\k)$ is the commutative ring of symmetric circulant matrices. Finally, we construct a canonical embedding
$\Aut(\k[\R_{2n}])\hookrightarrow\GL_n(\k)\times\Aut(\k[\R_n])$ over suitable rings $\k$ for each even integer $n\geq 2$ (Theorem \ref{prop:embedding}). This addresses \cite[Prob.\ 6.5]{BPS-1}.

In Section \ref{commutator-width}, we use a computer search to compute commutator widths of quandle rings $\F_q[\R_n]$ (Theorem \ref{Thm:cw-R_n}) and $\F_q[Q]$ (Theorems \ref{Thm: cw-Q} and \ref{thm:cw-8}) for certain values of $n$ and $q$, where $Q$ is a noncommutative quandle of order at most $7$. In particular, we provide the first known examples of quandle algebras of commutator width greater than $1$.

\section{Preliminaries}\label{sec-prelim}
 Recall that a {\it magma} is a set $Q$ equipped with a binary operation $(x,y) \mapsto x * y$, and a \emph{magma homomorphism} is a map between magmas that preserves the binary operations~$*$. We say that a magma $(Q,*)$ is a \emph{quandle} if the following axioms hold:
\begin{enumerate}[(Q1)]
\item $x*x=x$ for all $x \in Q$,
\item\label{Q2} For all $x \in Q$, the right-multiplication map $S_x\colon Q \to Q$ defined by $y \mapsto y\ast x$ is a magma automorphism of $(Q,\ast)$. Equivalently, each $S_x$ is a permutation, and $*$ is right-distributive.
\end{enumerate}

A magma satisfying only axiom \ref{Q2} is called a {\it rack}. Many interesting examples of quandles come from groups. This paper studies the following families of quandles.
\par

\begin{itemize}
\item 
A quandle  $Q$ is called {\it trivial} if $x*y=x$ for all $x, y \in Q$.  Unlike with groups, a trivial quandle can have arbitrary number of elements. We denote the $n$-element trivial quandle by $\T_n$ and an arbitrary trivial quandle by $\T$.
\item If $G$ is a group, then the binary operation $x*y= y^{-1} x y$ turns $G$ into the quandle $\Conj(G)$ called the {\it conjugation quandle} of $G$. More generally, the \emph{$m$-fold conjugation quandle} $\Conj_m(G)$ with $m\in\Z$ is defined by $x*y=y^{-m}xy^m$.
\item A group $G$ with the binary operation $x*y= y x^{-1} y$ turns the set $G$ into the quandle $\Core(G)$ called the {\it core quandle} of $G$. In particular, for the cyclic group $G= \mathbb{Z}_n$ of order $n$, we call $\Core(\Z_n)$ the {\it dihedral quandle} and denote it by $\R_n$.
\item Let $G$ be a group and $\varphi \in \Aut(G)$. Then the set $G$ with binary operation $x * y = \varphi(xy^{-1})y$ forms a quandle $\Alex(G, \varphi)$ called the  {\it generalized Alexander quandle} of $G$ with respect to $\varphi$. If $G$ is an additive abelian group, the quandle $\Alex(G, \varphi)$ defined by $$x\ast y= \varphi(x)+(\id-\varphi)(y),$$  is called an {\it Alexander quandle}. 
\item In particular, let $M$ be a $\Z[1/2]$-module, and let $\phi\in\Aut(M)$ denote multiplication by $1/2$. Then $\Alex(M,\phi)$ is called a \emph{midpoint quandle}. In particular, $(\Z_{2n+1},*)$ is called a \emph{cyclic midpoint quandle} and denoted by $C_{2n+1}$. See \cite{LT1} for a reference.
\end{itemize}

Let us recall some special classes of quandles.
\begin{enumerate}
\item A quandle $(Q,\ast)$ is {\it commutative} if $x\ast y=y\ast x$. For example, the dihedral quandle $\R_3$ is commutative. More generally, every midpoint quandle is commutative.
 \item  A quandle $(Q,\ast)$ is {\it medial} if $(x\ast y)\ast(t\ast z)=(x\ast t)\ast(y\ast z)$ for any $x, y, t, z \in Q$. In \cite{LT1}, it is shown that medial commutative quandles are precisely the midpoint quandles.
 
\item  A quandle $Q$ is called \emph{semi-Latin} (resp.\ {\it Latin}) if, for all $x\in Q$, the map \[L_x \colon Q \to Q,\qquad L_x(y) \coloneq x *y\] is surjective (resp.\ bijective). For example, if $M$ is an additive abelian group, then the Alexander quandle $\operatorname{Alex}(M,\varphi)$ is Latin if and only if $\operatorname{id}-\varphi$ is an automorphism of $M$. Note that every commutative quandle is Latin.
\end{enumerate}

Next, we focus on some definitions and results from \cite{BPS}. Let $\Bbbk$ be a unital associative commutative ring. From now on, except in situations where there is more than one binary operation on a set, we denote the multiplication in a quandle (resp. rack) by $(x,y) \mapsto xy$.

Given a quandle $Q$, let  $\Bbbk[Q]$ be the free $\k$-module generated by $Q$:
$$
\Bbbk[Q] =\Big\{ \sum_{i=1}^n\alpha_i x_i \mid \alpha_i \in \Bbbk,~ x_i \in Q,~n<\infty \Big\}.
$$
Make $\k[Q]$ into a nonunital nonassociative $\k$-algebra via the multiplication  $$\big(\sum_i\alpha_i x_i\big) \big(\sum_j\beta_j x_j\big) \coloneq\sum_{i,j}\alpha_i\beta_j (x_i x_j).$$
We call $\k[Q]$ the \emph{quandle algebra} of $Q$ with coefficients in $\Bbbk$. Note that $\Bbbk[Q]$ is associative if and only if $Q$ is a trivial quandle. If $Q$ is a rack, then its rack algebra $\Bbbk[Q]$ is defined analogously.
\par

The \emph{augmentation map} is the surjective $\k$-algebra homomorphism
$$
\varepsilon \colon \Bbbk[Q] \to \Bbbk
$$
defined by the formula  $$\varepsilon \big(\sum_i\alpha_i x_i\big)\coloneq \sum_i\alpha_i .$$
The \emph{augmentation ideal} of $\k[Q]$ is the two-sided ideal $\Delta_{\Bbbk}(Q) = \ker(\varepsilon)$. It is easy to see that $\{x-y\mid x, y \in Q \}$ is  a generating set for $\Delta_{\Bbbk}(Q)$ as a $\k$-module. Further, if $x_0 \in Q$ is a fixed element, then the set $\big\{x-x_0 \mid x \in Q \setminus \{ x_0\} \big\}$ is a basis for $\Delta_{\Bbbk}(Q)$ as an $\Bbbk$-module. For convenience, we denote $\Delta_\mathbb{Z}(Q)$ by $\Delta(Q)$. 
\par

We recall a result that characterizes trivial quandles in terms of their augmentation ideals. 
\par

\begin{theorem}[{\cite[Thm.\ 3.5]{BPS}}] \label{deltasqzero}
A quandle $Q$ is trivial if and only if $\Delta_{\Bbbk}^2(Q)=\{0\}$.
\end{theorem}

\section{Idempotents in quandle rings of commutative quandles}\label{Sec 4}

Given a rack $Q$ and a unital associative commutative ring $\k$, let $I(\k[Q])\subseteq \k[Q]$ denote the set of idempotent elements of $\k[Q]$. 
The following proposition was proven in \cite{BPS-1}.

\begin{proposition} \label{trivial-quandle-idempotents}
If $\T$ is a trivial quandle, then $ I(\k[\T]) =  x_0 + \Delta_\k (\T)$, where  $x_0 \in \T$ is a fixed element.
\end{proposition}

Given a rack $Q$ and a unital associative commutative ring $\k$, we say that an element $u\in\k[Q]$ is \emph{trivial} if $u\in Q$; in particular, $u\in I(\k[Q])$ and $\epsilon(u)=1$. It is natural to formulate the following conjecture.

 \begin{conjecture}[{\cite[Conj.\ 3.10]{MBMD}}]\label{conjid}    
For a semi-Latin quandle $Q$, the only idempotents of $\mathbb{Z}[Q]$ are the trivial elements. That is, $I(\Z[Q])=Q$.
\end{conjecture}

For example, dihedral quandles of odd orders are Latin. The conjecture is known to be true for $\R_3$ and $\R_5$; see \cite{MBMD}.

If we replace $\Z$ with a field, then the conjecture does not hold. 
Indeed, we strengthen \cite[Prop.\ 11.18]{BES} as follows.

\begin{proposition}\label{prop:unital}
    Let $\k$ be a unital commutative ring, and let $Q$ be a finite quandle of order $n$. Then the element
    \[
    e\coloneq \sum_{x\in X} x
    \]
    satisfies $ex=\epsilon(x)e$ for all $x\in \k[Q]$. If in addition $Q$ is Latin, then the following also hold:
    \begin{enumerate}
        \item For all $x\in\k[Q]$, we have $xe=\epsilon(x)e$.
        \item If $n\in\k^\times$, then $\frac{1}{n}e$ is idempotent.
    \end{enumerate}
\end{proposition}

\begin{proof}
    To prove the first claim, it suffices to show that $ex=e$ for all $x\in Q$; to prove the second claim, it suffices to show that $xe=e$ for all $x\in Q$ when $Q$ is Latin. Since $Q$ is finite, the first equality follows from quandle axiom Q2, and the second holds because $Q$ is Latin. The third claim is proven using the same direct calculation used to prove \cite[Prop.\ 11.18]{BES}.
\end{proof}

\subsection{Results}
In this section, we investigate idempotents of the quandle ring $\k[Q]$ for a commutative (hence Latin) quandle $Q$, where $\k$ is an integral domain, i.e., an associative and commutative ring with unity and without zero divisors. Let us prove that a commutative quandle does not contain a nontrivial subquandle with more than one element.
\begin{lemma}
\label{pro2.1}
Let $Q$ be a commutative quandle with more than one element. Then $Q$ is not a trivial quandle.
\end{lemma}
\begin{proof}
    Suppose that $\{x,y\}$ is a trivial subquandle of $Q$. We have $$S_x(x)=x=S_y(x)=S_x(y)$$ which implies that $x=y$.
\end{proof}
Now we introduce the notion of the length of an element of a quandle ring, which we study in the following theorem.
\begin{definition}[{\cite[Def.\ 2.6]{NP}}]
    Let $u=\sum\alpha_i e_i\in \mathbb{Z}[Q]$. We define the \emph{length} of $u$ to be the cardinality \[l(u)\coloneq|\{i: \alpha_i\neq 0\}|.\]
\end{definition} 
\begin{theorem}\label{Theorem 4.6}
   Let $\k$ be an integral domain with $\chara\k\neq 2$, and let $Q$ be a commutative quandle. Then:
   \begin{enumerate}
       \item $\k[Q]$ does not contain nontrivial idempotents of length 2. 
       \item If $u\in\k[Q]$ and $l(u)=3$, then $u$ is a nontrivial idempotent if and only if both of the following conditions hold:
       \begin{itemize}
           \item $3\in\k^\times$.
           \item There exists a subquandle $\{e_1,e_2,e_3\}\subseteq Q$ isomorphic to $\R_3$ such that $u$ equals $e\coloneq\frac{1}{3}(e_1+e_2+e_3)$ or $e_i-\frac{1}{3}e$ with $i=1,2,3$.
       \end{itemize}
   \end{enumerate}
   
\end{theorem}
\begin{proof}
Since $Q$ is commutative, $Q$ does not contain a trivial subquandle with more than one element by Lemma \ref{pro2.1}.  

Let $u=\sum_{i=1}^{n}\alpha_i e_i$ be an idempotent of $\k[Q]$. If $n=2$, then
\begin{align*}
   u^2&=(\alpha_{1}e_1+\alpha_{2}e_2)(\alpha_{1}e_1+\alpha_{2}e_2)\\
   &=\alpha_{1}^{2}e_1+\alpha_{1}\alpha_{2}e_1e_2+\alpha_{1}\alpha_{2}e_2e_1+\alpha_{2}^{2}e_2\\
   &=\alpha_{1}^{2}e_1+\alpha_{1}\alpha_{2}e_{1\ast 2}+\alpha_{1}\alpha_{2}e_{2\ast1}+\alpha_{2}^{2}e_2.
\end{align*}
Clearly, $e_{1\ast 2}\neq e_2$ and $e_{2\ast1}\neq e_1$. Further, $e_{1\ast 2}\neq e_1$ and $e_{2\ast1}\neq e_2$ since $Q$ has no trivial subquandle.
Thus $u^2=u$ if and only if 
$$\alpha_{1}^{2}-\alpha_{1}
=\alpha_{2}^{2}-\alpha_{2}
     = \alpha_{1}\alpha_{2}=0.$$
         Hence, either $\alpha_{1}=\alpha_{2}=0$, $\alpha_{1}=1$ and $\alpha_{2}=0$, or  $\alpha_{1}=0$ and $\alpha_{2}=1$.
         
       For $n=3$, first suppose that $\{e_1, e_2, e_3\}$ is a subquandle. Since $\{e_1, e_2, e_3\}$ is nontrivial, it must be isomorphic to $\R_3$. If $\epsilon(u)=1$, then Case 2 of the proof of \cite[Prop.\ 4.3]{BPS-1} shows that either $u$ is trivial or $3\in\k^\times$ and $\alpha_1=\alpha_2=\alpha_3=\frac{1}{3}$. Otherwise, $\epsilon(u)=0$, so Case 1 of the proof of \cite[Prop.\ 4.3]{BPS-1} states that
       \[
       \alpha_3=-(\alpha_1+\alpha_2),\qquad \alpha_1-\alpha_2=3(\alpha_1-\alpha_2)(\alpha_1+\alpha_2).
       \]
       It follows that
       \[
       (3\alpha_3+1)(\alpha_2-\alpha_1)=0.
       \]
       Since $\epsilon(u)=0$, we deduce that either $u$ is trivial or $3\in\k^\times$ and $u$ has one of the proposed forms. 
       Conversely, the last part of Proposition \ref{prop:unital} and direct calculations show that the sums in the claim are idempotent.
       
       Now, assume that $\{e_1, e_2, e_3\}$ is not a subquandle. Then
       \begin{align*}
           u^2&=(\alpha_{1}e_1+ \alpha_{2}e_2+ \alpha_{3}e_3)(\alpha_{1}e_1+ \alpha_{2}e_2+ \alpha_{3}e_3)\\
           &=\alpha^{2}_{1}e_1+2\alpha_{1}\alpha_{2}e_{1\ast 2}+\alpha^{2}_{2}e_2+2\alpha_{2}\alpha_{3}e_{2\ast3}+\alpha^{2}_{3}e_3+2\alpha_{1}\alpha_{3}e_{1\ast3}.
       \end{align*} 
       We have the following three cases.
        \begin{itemize}
            \item Case 1: Suppose that $e_{1\ast2}=e_3$, $e_{2\ast3}=e_1$ and $e_{1\ast3}\neq e_2$. Then $u^2=u$ if and only if 
            \begin{align*}
               \alpha_{1}&= \alpha^{2}_{1}+2\alpha_{2}\alpha_{3},\\
               \alpha_{2}&=\alpha^{2}_{2},\\
               \alpha_{3}&=\alpha^{2}_{3}+2\alpha_{1}\alpha_{2},\\
               0&=\alpha_{1}\alpha_{3}.
            \end{align*}
             This system of equations has solutions: $(0,0,0)$, $(0,0,1)$, $(0,1,0)$ and $(1,0,0)$. 
           \item Case 2: Suppose that $e_{1\ast2}=e_3$, $e_{2\ast3}\neq e_1$ and $e_{1\ast3}\neq e_2$. In this case, $u^2=u$ if and only if 
            \begin{align*}
               \alpha_{1}&= \alpha^{2}_{1},\\
               \alpha_{2}&=\alpha^{2}_{2},\\
               \alpha_{3}&=\alpha^{2}_{3}+2\alpha_{1}\alpha_{2},\\
        0&=\alpha_{1}\alpha_{3}=\alpha_{2}\alpha_{3}.
            \end{align*}
           It is easy to prove that this system has the same solutions as those of Case 1.
            \item Case  3: If $e_{1\ast2}\neq e_3$, $e_{2\ast3}\neq e_1$ and $e_{1\ast3}\neq e_2$, then $u^2=u$ implies that 
            \begin{align*}
               \alpha_{1}&= \alpha^{2}_{1},\\
               \alpha_{2}&=\alpha^{2}_{2},\\
               \alpha_{3}&=\alpha^{2}_{3},\\
            0&=\alpha_{1}\alpha_{2}=\alpha_{1}\alpha_{3}=\alpha_{2}\alpha_{3}.
            \end{align*}
This system has the same solutions as the previous two.
        \end{itemize}
\end{proof}
\begin{corollary}
 Let $G$ be a group. Then: 
   \begin{enumerate}
       \item 
    If $G$ has exponent $3$, then $\mathbb{Z}[\operatorname{Core}(G)]$ does not contain nontrivial idempotents of length  $n=2,3$.
    \item If $m\in\Z$ and  $\operatorname{Conj}_m(G)$ is commutative, then $\mathbb{Z}[\operatorname{Conj}_{m}(G)]$ does not contain nontrivial idempotents.
    \end{enumerate} 
\end{corollary}
\begin{proof}
1) It is shown in \cite[Prop.\ 3.1]{LT} that $G$ has exponent $3$ if and only if $\operatorname{Core}(G)$ is commutative. Hence, Theorem \ref{Theorem 4.6} yields the claim.

2) It is easy to prove that $\operatorname{Conj}_m(G)$ is commutative if and only if $G$ is trivial.
\end{proof}

Next, we give an additional assumption (namely, orderability) under which Theorem \ref{Theorem 4.6} can be strengthened. This verifies Conjecture \ref{conjid} for all ordered commutative quandles.

\begin{theorem}\label{Theorem 4.8}
    Let $\k$ be an integral domain with $\chara\k\neq 2$. Let $Q$ be an ordered commutative quandle, i.e., $x<y \Rightarrow x*z<y*z$ for all $x,y,z \in Q$. Then $\k[Q]$ does not contain nontrivial idempotents.
\end{theorem}
\begin{proof}
   Let $u=\sum_{i=1}^{n}\alpha_i x_i$ be an idempotent of $\k[Q]$ where $\alpha_i\in\k\setminus\{0\}$ and $x_1<\cdots<x_n$.
We have $$u^{2}=\left(\sum_{i=1}^{n}\alpha_i x_i\right)\left(\sum_{j=1}^{n}\alpha_j x_j\right)=\sum_{i=1}^{n}\sum_{j=1}^{n}\alpha_i\alpha_j (x_i\ast x_j)$$
  Since $Q$ is an ordered commutative quandle, $x_i \ast x_j < x_n \ast x_n = x_n$ for $i < n$ or $j < n$. Thus, the coefficient of $x_n$ in $u^2$ is $\alpha_n^2$. Therefore, the equality $u^2 = u$ implies that $\alpha_n^2 = \alpha_n$, i.e., $\alpha_n = 1$. 

  By strict ordering and commutativity, the elements immediately below $x_n$ in $u^2$ are $x_n \ast x_{n-1}$ and $x_{n-1} \ast x_n$. 
  Since $x_n \ast x_{n-1} = x_{n-1} \ast x_n$, the coefficient of $x_{n-1} \ast x_n$ in $u^2$ is $2\alpha_{n-1} \alpha_n$. 
  On the other hand, since each $x_i$ is idempotent, we deduce from orderability that
  \[
  x_{n-1}<x_{n-1}*x_n<x_n.
  \]
  It follows that $x_n\ast x_{n-1}\neq x_i$ for all $1\leq i\leq n$, so the coefficient of $x_{n-1}*x_n$ in $u=u^2$ is $0$. Hence,
  \[
  0=2\alpha_n \alpha_{n-1}=2\alpha_{n-1},
  \]
  so $\alpha_{n-1} = 0$. This contradicts our assumption that $\alpha_i\neq 0$ for each $i$.  
\end{proof}

\begin{corollary}
    Let $\k$ be an integral domain with $\chara\k\neq 2$, let $A$ be an ordered $\Z[1/2]$-algebra, and let $Q=\Alex(A,1/2)$ be the midpoint quandle over $A$. Then $\k[Q]$ contains no nontrivial idempotents.
\end{corollary}

In \cite[Question 3.1]{BE}, it was asked whether for any unital ring $\Bbbk$ without zero divisors and infinite field $\mathbb F$, the quandle ring $\Bbbk[\operatorname{Core}(\mathbb F)]$ has only trivial idempotents. The following proposition proves that the question has a negative answer. Indeed, let $\k=\F\coloneq\mathbb{F}_2(X)$ be the field of rational functions over $\mathbb{F}_2$. Then $\F=(\mathbb{F}_2(X), +)$ is a group of exponent $2$, so the following proposition states that $\k[\Core(\F)]=\mathbb{F}_2(X)[\operatorname{Core}(\mathbb{F}_2(X))]$ has a nontrivial idempotent.
\begin{proposition}
   Let $\Bbbk$ be an integral domain, and let $G$ be a nontrivial group. Then:
   \begin{enumerate}
       \item $\Bbbk[\operatorname{Conj}_{m}(G)]$ has nontrivial idempotents.
        \item If $G$ has exponent $2$,  then $\Bbbk[\operatorname{Core}(G)]$ has nontrivial idempotents. 
   \end{enumerate}
   \end{proposition}
\begin{proof}
Let $e$ be the identity element of $G$. For $x \in G$ such that $x \neq e$, it is easy to see that $\{x, e\}$ is a trivial subquandle of $\operatorname{Conj}_m(G)$. Moreover, if $G$ has exponent $2$, then $\{x, e\}$ is also a trivial subquandle of $\operatorname{Core}(G)$. Hence, \cite[Prop.\ 3.3]{MBMD} yields the claim.
\end{proof}


\section{Automorphisms of quandle algebras}\label{sec-auto}

Given a quandle $Q$, denote by $\Aut (\k[Q])$ the group of $\k$-algebra automorphisms of $\k[Q]$. 
In this section, we compute $\Aut (\k[Q])$ for all trivial quandles and dihedral quandles of odd orders, with partial results for dihedral quandles of even orders. 

Evidently, $\Aut (Q) \leq \Aut (\k[Q])$. Furthermore, if $Q$ is a finite quandle with $n$ elements, then $\Aut (\k[Q]) \leq \mathrm{GL}_n(\k)$. 
Note that any $\phi \in \Aut (\k[Q])$ is defined by its action on elements of $Q$. If $Q = \{x_1, x_2, \ldots, x_n \}$, then each $\phi(x_i)$ is an idempotent of $\k[Q]$, and the quandle $\phi(Q)$ is isomorphic to $Q$. In \cite{BPS-1, MBMD}, these facts were used to prove the following.

\begin{theorem}[{\cite[Cor. 5.5]{MBMD}}] Let $n\geq 1$, let $\mathrm{FQ}_n$ be the free quandle of rank $n$, and let $\mathrm{WB}_n$ be the welded braid group on $n$ strands. Then 
  $$\Aut( \mathbb{Z}[\mathrm{FQ}_n]) \cong \Aut(\mathrm{FQ}_n)\cong \mathrm{WB}_n.$$
\end{theorem}

\begin{theorem}[{\cite[Thm. 6.1, Prop. 6.3, Thm. 6.4]{BPS-1}}]\label{thm:t2}\phantom{-}
\begin{enumerate}
    \item $\Aut (\mathbb{Z}[\T_2]) \cong \mathbb{Z} \rtimes \mathbb{Z}_2$.
    \item $\Aut (\mathbb{Z}[\R_3]) \cong \Sigma_3 \cong \mathrm{\Aut }(\R_3)$.
    \item $\Aut (\mathbb{Z}[\R_4]) \cong ( \mathbb{Z}_2 \times \mathbb{Z}_2) \rtimes \mathbb{Z}_2$.
\end{enumerate}
\end{theorem}

In this section, we consider \cite[Prob.\ 6.5]{BPS-1} for finite quandles.

\begin{problem}[{\cite[Prob.\ 6.5]{BPS-1}}]\label{prob:aut}
Compute the automorphism groups of integral quandle algebras of all finite trivial quandles and dihedral quandles.
\end{problem}

\subsection{Trivial quandles}
Let $\k$ be a unital associative commutative ring. Recall that, for each $n\geq 0$, the \emph{general affine group} $\Aff_n(\k)$ is defined to be the obvious semidirect product $\k^n\rtimes\GL_n(\k)$. Equivalently, $\Aff_n(\k)$ is the group of invertible affine transformations of the affine space $\mathbb A^n_\k=\k^n$. Some connections between $\Aff_n(\k)$ and rack theory can be found in \cite{LT2}.

\begin{theorem}\label{thm:aff} Let $\k$ be a unital associative commutative ring, and let $\T_n=\{x_1,\dots x_n\}$ be the trivial quandle of order $n\geq 1$. Then
     $$\Aut(\k[\T_{n}])\cong\Aff_{n-1}(\k).$$
\end{theorem}

\begin{proof}
It is shown in \cite[Sec.\ 6]{BPS-1} that for every automorphism $\phi \in \Aut (\k[\T_n])$, there exist constants $(\alpha_{j,i})_{1\leq j\leq n-1,1\leq i\leq n}$ in $\k$ such that
\[
\phi(x_i)=\left(1 -\sum^{n-1}_{j=1}\alpha_{j,i}  \right) x_1+ \sum^{n-1}_{j=1} \alpha_{j,i} x_{j+1}
\]
for all $1\leq i\leq n$. 
Since $\k[\T_n]\cong\k^n$ as $\k$-modules, $\phi$ can be written uniquely as a matrix of the form
\[
\phi =
   \begin{pmatrix}
1 -\sum^{n-1}_{j=1}\alpha_{j,1} & 1 -\sum^{n-1}_{j=1}\alpha_{j,2} & \cdots & 1 -\sum^{n-1}_{j=1}\alpha_{j,n} \\
\alpha_{1,1} & \alpha_{1,2} & \cdots & \alpha_{1,n} \\
\vdots & \vdots & \ddots & \vdots \\
\alpha_{n-1,1} & \alpha_{n-1,2} & \cdots & \alpha_{n-1,n} 
\end{pmatrix}  \in\GL_n(\k).
\]
Conversely, every matrix $\phi$ of this form is a $\k$-algebra automorphism of $\Aut(\k[\T_n])$. Indeed, for all $a,b\in \k[\T_n]$, we have $ab=\epsilon(b)a$ and $\epsilon(\phi(b))=\epsilon(b)$, so 
\[
\phi(ab)=\phi(\epsilon(b)a)=\epsilon(b)\phi(a)=\epsilon(\phi(b))\phi(a)=\phi(a)\phi(b).
\]
Hence, an invertible matrix lies in $\Aut(\k[\T_n])\leq\GL_n(\k)$ if and only if each of its columns sums to $1$. This is precisely the description of one of the canonical embeddings $\Aff_{n-1}(\k)\hookrightarrow\GL_n(\k)$; see, for example, \cite[Thm.\ B]{poole}.\footnote{Although the paper \cite{poole} works over a field, its proof of Theorem B works just as well over any unital associative commutative ring.}
\end{proof}

Of course, the $n=1$ case of Theorem \ref{thm:aff} recovers the obvious fact that $\Aut(\k[\T_1])=1$, while the $n=2$ case recovers the first part of Theorem \ref{thm:t2}.

\subsection{Dihedral quandles} First, we study $\Z[\R_n]$ when $n$ is odd.

\begin{lemma}\label{lem:odd}
  For all odd $n\geq 1$, the dihedral quandle $\R_n$ is Latin.  
\end{lemma}
\begin{proof}
    Let $S_a$ denote the left multiplication by an element $a$ of $\R_n$. For $b \in \R_n$, the equality $L_a(x)=a*x = b$ is equivalent to $2x - a \equiv b \pmod{n}$. Since $n$ is odd, $2$ is invertible modulo $n$. Therefore, $x \equiv 2^{-1}(a + b) \pmod{n}$. Hence, $L_a$ is bijective.
\end{proof}
  Therefore, the following proposition holds.
 \begin{proposition}\label{prop:odd-dihedral}
     Let $n\geq 1$ be odd, and let $\R_n$ be the dihedral quandle with underlying set $\Z_n$. Then there are canonical isomorphisms
     \[\Aut (\mathbb{Z}[\R_n])\cong \mathrm{\Aut }(\R_n)\cong\Aff_1(\Z_n).\]
 \end{proposition}
 \begin{proof}
     An upcoming article of Elhamdadi, Virgin, and the second author \cite{ETV} verifies Conjecture \ref{conjid} for Latin dihedral quandles. Therefore, the first and second isomorphisms are \cite[Prop.\ 2.4]{BE} and \cite[Thm.\ 2.1]{Elhamdadi2012}, respectively.
 \end{proof}

 Here is another (unrelated) consequence. Write $\R_n=\{a_0,\dots,a_{n-1}\}$. 
Given a nonassociative ring $A$, recall that the \emph{left nucleus} of $A$ is defined to be the associative subring
\[
\Nuc_l(A)\coloneq\{a\in A\mid a(xy)=(ax)y\text{ for all }x,y\in A\}.
\]

\begin{proposition}
    Let $\k$ be a unital commutative ring, let $n\geq 1$ be an odd integer, and let $e\coloneq\sum_{i\in\Z_n}a_i$ be the element of $\Z[\R_n]$ constructed in Proposition \ref{prop:unital} (cf.\ Lemma \ref{lem:odd}). Then the left nucleus of $\k[\R_n]$ is $\Nuc_l(\k[\R_n])=\k e$.
\end{proposition}

\begin{proof}
    Certainly, $\k e$ is contained in the $\Nuc_l(\k[\R_n])$. Conversely, let $a\in\Nuc_l(\k[\R_n])$, and write $x=\sum_{i\in \Z_n}\alpha_ia_i$. We have to show that $\alpha_i=\alpha_j$ for all $i,j\in\Z_n$. Indeed, for all $j\in\Z_n$, 
    \[
    \sum_{i\in \Z_n}\alpha_ia_i=\sum_{i\in \Z_n}\alpha_ia_{2j-(2j-i)}=(xa_j)a_j=x(a_ja_j)=xa_j=\sum_{i\in \Z_n}\alpha_ia_{2j-i}.
    \]
    Since $n$ is odd, multiplication by $2$ is invertible in $\Z_n$, so the claim follows from the fact that the $a_i$'s are linearly independent in $\k[\R_n]$.
\end{proof}

Our next goal is to compute the automorphism group of $\k[\R_{2n}]$ for all odd $n\geq 3$; see Theorem \ref{thm:r2n}.

\subsubsection{Setup}

Given a commutative ring $\k$ and a quandle $Q$, form a nonunital nonassociative $\k$-algebra $A^\k_Q$ as follows. The underlying $\k$-module of $A_Q^\k$ is $\k[Q]\oplus \k[Q]$. The ring multiplication is 
\[
(a,b)(c,d)\coloneq(ac,bc).
\]
Note that $A^\k_Q$ is not the product in the category of (nonunital nonassociative) $\k$-algebras. 

\begin{lemma}\label{lem:ax}
    For all commutative rings $\k$ and racks $Q$, there is a canonical $\k$-algebra isomorphism 
\[
\psi\colon \k[Q\times \T_2]\bij A_Q^\k.
\]
\end{lemma}

\begin{proof}
    Write $\T_2=\{a,b\}$, and define $\psi$ by
\[
(x,a)\mapsto(x,0),\qquad (x,b)\mapsto (x,x).
\]
It is straightforward to check that $\psi$ is a $\k$-algebra isomorphism.
\end{proof}

\begin{remark}
    For all commutative rings $\k$ and racks $X$ and $Y$, there is a unique $\k$-algebra isomorphism
\[
\k[X\times Y]\cong \k[X]\otimes_{\k}\k[Y].
\]
This is straightforwardly checked using the universal property of the tensor product of (nonunital nonassociative) $\k$-algebras.
\end{remark}

\subsubsection{Symmetric circulant matrices}

In the following, let $\k$ be a unital associative commutative ring, let $n\geq 1$, and fix a $\k$-basis $\{e_1,\dots,e_n\}$ of $\k^n$. 
Let $\sigma,\rho\in\GL_n(\k)$ be the shift operator $\sigma(e_i)\coloneq e_{i+1}$ and the reversal operator $\rho(e_i)\coloneq e_{-i}$, where the indices are considered modulo $n$. 
Recall that $n\times n$ matrices of the form
\[
\sum^{n-1}_{i=0}c_{i+1}\sigma^i=\begin{pmatrix}
    c_1&c_2&\cdots &c_n\\
    c_n&c_1&\cdots&c_{n-1}\\
    \vdots&\vdots&\ddots&\vdots\\
    c_2&c_3&\cdots&c_1
\end{pmatrix}\in M_n(\k)
\]
are called \emph{circulant}; see \cite{circulant} for a reference. 

Let $\SC_n(\k)$ be the commutative $\k$-subalgebra of symmetric circulant matrices in $M_n(\k)$. Note that $\SC_n(\k)\cong\k^{\lfloor n/2\rfloor+1}$ as $\k$-modules.  
To compute $\Aut(\k[\R_{2n}])$ when $n$ is odd, we need the following facts about $\SC_n(\k)$.

\begin{lemma}\label{lem:scnk}
    Let $\k$ be a unital associative commutative ring, and let $n\geq 1$. Then $\SC_n(\k)$ is precisely the centralizer of $\sigma$ and $\rho$ in $M_n(\k)$.
\end{lemma}

\begin{proof}
    The $\k$-algebra of circulant matrices is the centralizer of $\sigma$ in $M_n(\k)$ (see \cite[Thm.\ 3.1.1]{circulant}), and likewise for centrosymmetric matrices and $\rho$. The reader can verify that a circulant matrix is symmetric if and only if it is centrosymmetric.
\end{proof}

Interpret $\Aff_1(\Z_n)$ as the group of invertible affine transformations $T_{a,b}(x)\coloneq ax+b$ of the affine line $\mathbb A^1_{\Z_n}=\Z_n$, and let $\Aff_1(\Z_n)$ act on the basis $\{e_1,\dots,e_n\}$ of $\k^n$ in the obvious way. For each affine transformation $T_{a,b}\in\Aff_1(\Z_n)$, let $P_{a,b}\in\GL_n(\k)$ be the corresponding permutation matrix. 
\begin{lemma}\label{lemma:conjugation}
    Let $\k$ be a unital associative commutative ring, and let $n\geq 1$. Then the ring of symmetric circulant matrices $\SC_n(\k)$ is stable under the right action of $\Aff_1(\Z_n)$ on $M_n(\k)$ by conjugation. That is,
    \[
    M^{T_{a,b}} \coloneq P_{a,b}MP_{a,b}\inv\in\SC_n(\k)
    \]
    for all $M\in\SC_n(\k)$ and $T_{a,b}\in\Aff_1(\Z_n)$. In particular, the action extends to an action of $\Aff_1(\Z_n)$ on $\SC_n(\k)\rtimes\SC_n(\k)^\times=\Aff_1(\SC_n(\k))$ by conjugation on both coordinates.
\end{lemma}

\begin{proof}
    Given a symmetric circulant matrix \[
    M=\sum^{n-1}_{i=0}c_i\sigma^i,\qquad c_i=c_{-i},
    \]
    a direct calculation shows that
    \[
     M^{T_{a,b}}=\sum^{n-1}_{i=0}c_i\sigma^{ai},\qquad c_i=c_{-i}.
    \]
    Since $a$ is invertible in $\Z_n$, the matrix $M^{T_{a,b}}$ is symmetric circulant.
\end{proof}

\subsubsection{Automorphism group calculation} 
Let $n\geq 1$, and write $\R_n=\{a_0,\dots,a_{n-1}\}$. We begin with a lemma. Given a unital associative commutative ring $\k$, consider the $\k$-algebra $\operatorname{End}_{\k\text{-}\mathsf{Mod}}(\k[\R_n])$ of $\k$-module endomorphisms of $\k[\R_n]$. In this $\k$-algebra, let $E^\k_{\R_n}$ be the centralizer of the right multiplication maps $S_{a_0},\dots,S_{a_{n-1}}$. Then
\begin{equation}\label{eq:ekrn}
    E^\k_{\R_n}=\{f\in\operatorname{End}_{\k\text{-}\mathsf{Mod}}(\k[\R_n])\mid f(xy)=f(x)y\text{ for all }x,y\in\k[\R_n]\}.
\end{equation}

\begin{lemma}\label{lemma:ern}
    For all odd $n\geq 1$, the canonical $\k$-module isomorphism $\k[\R_n]\cong\k^n$ induces a canonical $\k$-algebra isomorphism $E^\k_{\R_n}\cong \SC_n(\k)$.
\end{lemma}

\begin{proof}
    Identify $\k[\R_n]\cong\k^n$ via the basis $\{a_0,\dots,a_{n-1}\}$ of $\k[\R_n]$, and let $\sigma,\rho\in\GL_n(\k)$ be the shift and reversal operators defined earlier. Then
    \[
    S_{a_i}(a_j)=a_{2i-j}=(\sigma^{2i}\circ\rho)(a_j)
    \]
    for all $i,j\in\Z_n$. Since $n$ is odd, the subgroup of $\GL_n(\k)$ generated by the maps $\sigma^{2i}\circ\rho$ with $i\in\Z_n$ equals the (dihedral) subgroup generated by $\sigma$ and $\rho$. 
    Therefore, the claim follows from Lemma \ref{lem:scnk} and the definition of $E^\k_{\R_n}$.
\end{proof}

\begin{theorem}\label{thm:r2n}
    Let $\k$ be a unital associative commutative ring. Then for all odd integers $n\geq 1$, 
    there exists an action of $\Aut(\k[\R_n])$ on $\Aff_1(\SC_n(\k))$ such that
    \[
    \Aut(\k[\R_{2n}])\cong \Aff_1(\SC_n(\k))\rtimes \Aut(\k[\R_n]).
    \]
    In particular, 
\[
\Aut(\Z[\R_{2n}])\cong \Aff_1(\SC_n(\Z))\rtimes \Aff_1(\Z_n),
\]
where $\Aff_1(\Z_n)$ acts on $\Aff_1(\SC_n(\Z))$ by conjugation in the sense of Lemma \ref{lemma:conjugation}.
\end{theorem}

\begin{proof}
    The assignment of a group to its core quandle is functorial, so the group isomorphism $\Z_{2n}\cong\Z_n\times\Z_2$ descends to a quandle isomorphism $\R_{2n}\cong \R_n\times\R_2\cong \R_n\times \T_2$. Therefore, Lemma \ref{lem:ax} provides a $\k$-algebra isomorphism $\k[\R_{2n}]\cong\k[\R_n\times \T_2]\cong A_{\R_n}^\k$. So, it suffices to compute $\Aut(A_{\R_n}^\k)$.

    By Proposition \ref{prop:odd-dihedral} and Lemma \ref{lemma:ern}, it suffices to construct a group isomorphism
    \[
    \Phi\colon \Aff_1(E_{\R_n}^\k)\rtimes\Aut(\k[\R_n])\to\Aut(A_{\R_n}^\k),
    \]
    where $\Aut(\k[\R_n])$ acts on $\Aff_1(E_{\R_n}^\k)=E_{\R_n}^\k\rtimes (E_{\R_n}^\k)^\times$ by conjugation on both coordinates. In other words, the group operation of $\Aff_1(E_{\R_n}^\k)\rtimes\Aut(\Z[\R_n])$ is
    \[
    (f,\phi_1,\psi_1)(g,\phi_2,\psi_2)=(f+(\phi_1\circ g^{\psi_1}),\phi_1\circ \phi_2^{\psi_1},\psi_1\circ\psi_2),
    \]
    where the superscripts $h^{\psi_1}\coloneq\psi_1 \circ h\circ\psi_1\inv$ denote conjugation by $\psi_1$. 
    Define $\Phi$ by sending each triple $(f,\phi,\psi)$ to the $\k$-linear map
    \[
    \Phi_{f,\phi,\psi}\colon A_{\R_n}^\k\to A_{\R_n}^\k,\qquad (x,y)\mapsto(\psi(x),(f\circ\psi)(x)+(\phi\circ\psi)(y)).
    \]
    The reader can verify by routine (albeit tedious) calculation that $\Phi_{f,\phi,\psi}$ is a $\k$-algebra automorphism with two-sided inverse
    \[
    \Phi_{f,\phi,\psi}\inv(x,y)=(\psi\inv(x),(\psi\inv\circ\phi\inv)(y-f(x)))
    \]
    and that $\Phi$ is an injective group homomorphism. As a hint, injectivity is shown by evaluating elements of $\ker(\Phi)$ first at $(x,0)\in A_{\R_n}^\k$ and then at $(0,y)\in A_{\R_n}^\k$.

    It remains to show that $\Phi$ is surjective. Fix $\Psi\in\Aut(A_{\R_n}^\k)$. We have to find a triple $(f,\phi,\psi)$ such that $\Psi=\Phi_{f,\phi,\psi}$. 
    The right annihilator of $A_{\R_n}^\k$ is the two-sided ideal $I\coloneq 0\oplus \k[\R_n]$, so $\Psi(I)=I$. In particular, $\Psi$ descends to an automorphism $\psi$ of the quotient ring $A_{\R_n}^\k/I\cong\k[\R_n]$. As a result, we can define $\k$-module endomorphisms $\alpha,\beta\colon \k[\R_n]\to\k[\R_n]$ via the formulas \[\Psi(x,0)=(\psi(x),\alpha(x)),\qquad\Psi(0,y)=(0,\beta(y)).\] Then 
    $\beta$ is bijective, and the reader can verify directly that \[\Psi(x,y)=(\psi(x),\alpha(x)+\beta(y)),\quad\alpha(xy)=\alpha(x)\psi(y),\quad \beta(xy)=\beta(x)\psi(y)\] for all $x,y\in\k[\R_n]$. It follows from \eqref{eq:ekrn} that the composites
    \[
    f\coloneq \alpha\circ\psi\inv,\qquad \phi\coloneq \beta\circ\psi\inv
    \]
    lie in $E_{\R_n}^\k$, and $\phi$ is invertible because $\beta$ is invertible. 
    Hence, $\Psi=\Phi_{f,\phi,\psi}$ as desired.
\end{proof}

Of course, the $n=1$ case of Theorem \ref{thm:r2n} recovers the first part of Theorem \ref{thm:t2}.

Together, Proposition \ref{prop:odd-dihedral} and Theorems \ref{thm:aff} and \ref{thm:r2n} answer Problem \ref{prob:aut} except for the remaining open case.

\begin{problem}\label{prob:aut-even}
    Compute the automorphism group of $\Z[\R_{2n}]$ for all even integers $n\geq 2$.
\end{problem}

\subsection{A canonical embedding}
Although we do not have a complete answer to Problem \ref{prob:aut-even}, imitating the start of the proof of Theorem \ref{thm:r2n} leads to a partial description of $\Aut(\k[\R_{2n}])$ when $n\geq 2$ is even.

Let $\k$ be a unital associative commutative ring, and fix an integer $n\geq 1$. Then the natural projection $\Z_{2n}\twoheadrightarrow\Z_n$ descends to a surjective quandle homomorphism $\R_{2n}\twoheadrightarrow\R_n$ that, in turn, induces a surjective $\k$-algebra homomorphism $\pi\colon \k[\R_{2n}]\twoheadrightarrow\k[\R_n]$. If we write $\R_{2n}=\{a_0,\dots,a_{2n-1}\}$ and $I_\k\coloneq \ker\pi$, then
    \[
    I_\k=\operatorname{span}_\k\langle b_0,\dots,b_{n-1}\rangle=\operatorname{span}_\k\langle b_n,\dots,b_{2n-1}\rangle,\qquad b_i\coloneq a_i-a_{i+n},
    \]
    where the indices are considered modulo $2n$. 
    
    Evidently, $I_\k$ is the right annihilator of $\k[\R_{2n}]$, and $S_{a_j}(b_i)=b_{2j-i}$ for all $i,j\in\Z_{2n}$, where $S_{a_j}\colon\k[\R_{2n}]\to\k[\R_{2n}]$ denotes right multiplication by $a_j$. In particular, every automorphism $\phi$ of $\k[\R_{2n}]$ restricts to an automorphism $\phi|_{I_\k}$ of $I_\k$ and descends to a $\k$-algebra automorphism $\overline{\phi}$ of $\k[\R_{2n}]/I_\k\cong\pi(\k[\R_{2n}])\cong\k[\R_n]$. For example, if $n$ is even, then
    \[
   S_{a_{j+n/2}}(b_i)=b_{2j+n-i}=-b_{2j-i}=-S_{a_j}(b_i)
    \]
    for all $i,j\in\Z_{2n}$, so \begin{equation}\label{eq:saj}
        S_{a_{j+n/2}}|_I=-S_{a_j}|_I.
    \end{equation}
    In particular, 
    \begin{equation}\label{eq:saj2}
        S_{a_{j+n}}|_I=-S_{a_{j+n/2}}|_I=S_{a_j}|_I
    \end{equation}
    for all $j\in\k_{2n}$, so the right $\k[\R_{2n}]$-module structure on $I_\k$ factors through a right $\k[\R_n]$-module structure on $I_\k$.

    The following proposition states that, for all even $n\geq 2$ and for all suitable rings $\k$, the actions of $\Aut(\k[\R_{2n}])$ on $I_\k$ and $\k[\R_{2n}]/I_\k\cong\k[\R_n]$ completely determine $\Aut(\k[\R_{2n}])$ and do not interact with each other.

\begin{theorem}\label{prop:embedding}
    Let $\k$ be an integral domain with $\chara\k\neq 2$. Then for each even integer $n\geq 2$, there exists a canonical embedding
    \[
    \Aut(\k[\R_{2n}])\hookrightarrow\GL_n(\k)\times\Aut(\k[\R_n]). 
    \]
\end{theorem}

\begin{proof}
    The $\k$-basis $\{b_0,\dots,b_{n-1}\}$ of $I_\k$ canonically identifies $I_\k\cong\Z^n$ as $\k$-modules, so it suffices to prove that the canonical group homomorphism
    \[
    \Phi\colon \Aut(\k[\R_{2n}])\to \Aut_{\k\text{-}\mathsf{mod}}(I_\k)\times \Aut(\k[\R_{2n}]/I_\k),\qquad\phi\mapsto(\phi|_I,\overline{\phi})
    \]
    is injective. 
    
    Fix $\phi\in\ker\Phi$, and consider the $\k$-module endomorphism $\psi\coloneq\phi-\id_{\k[\R_{2n}]}$ of $\k[\R_{2n}]$. Since $\overline{\phi}=\id_{\k[\R_{2n}]/I_\k}$, the image of $\psi$ is contained in $I_\k$. On the other hand, $\phi|_{I_\k}=\id|_{I_\k}$, so $\psi$ factors through a unique $\k$-module homomorphism $\overline{\psi}\colon\k[\R_{2n}]/I_\k\to I_\k$. 
    It suffices to show that $\overline\psi\equiv 0$. For all $i\in\Z_{2n}$, let $c_i\coloneq\overline\psi(\pi(a_i))\in I_\k$. Since $I_\k\cong\k^n$ as $\k$-modules, it suffices to show that $2c_i=0$ for all $i\in\Z_{2n}$. 
    
    For all $x,y\in\k[\R_{2n}]$, we first compute
    \begin{align*}
        xy+\overline\psi(\pi(xy))&=xy+\psi(xy)\\
        &=\phi(xy)\\
        &=\phi(x)\phi(y)\\
        &=(x+\psi(x))(y+\psi(y))\\
        &=xy+x\psi(y)+\psi(x)y+\psi(x)\psi(y)\\
        &=xy+\psi(x)y\\
        &=xy+\overline\psi(\pi(x))y,
    \end{align*}
    so $\overline\psi(\pi(xy))=\overline\psi(\pi(x))y$. In the third equality, we used the assumption that $\phi$ is a ring endomorphism; in the sixth equality, we used the fact that $I_\k\supseteq\psi(\k[\R_{2n}])$ is the right annihilator of $\k[\R_{2n}]$. For all $i,j\in\Z_{2n}$, we deduce that
    \[
    S_{a_j}(c_i)=c_ia_j=\overline\psi(\pi(a_i))a_j=\overline\psi(\pi(a_ia_j))=c_{2j-i}
    \]
    and, hence,
    \[
    S_{a_{i+n/2}}(c_i)=c_{i+n}=\overline\psi(\pi(a_{i+n}))=\overline\psi(\pi(a_i))=c_i.
    \]
    Thus, taking $j\coloneq i+n/2$ and applying \eqref{eq:saj} yields
    \[2c_i=S_{a_i}(c_i)+S_{a_{i+n/2}}(c_i)=S_{a_i}(c_i)-S_{a_i}(c_i)=0\]
    for all $i\in\Z_{2n}$, as desired.
\end{proof}

\begin{remark}
    If we attempt the same proof of Theorem \ref{prop:embedding} for odd $n\geq 1$, then we find that $\ker\Phi\cong\k^{(n+1)/2}\cong\SC_n(\k)$ as $\k$-modules; cf.\ Theorem \ref{thm:r2n}.
\end{remark}

\begin{lemma}\label{lem:injective}
    In the above setting, take $\k=\Z$ and $n=4$. Then the four maps $\{\pm S_{a_0}|_I,\pm S_{a_1}|_I\}$ are pairwise distinct.
\end{lemma}

\begin{proof}
    For each $i=0,1$, we have $S_{a_i}|_I\neq-S_{a_i}|_I$ because $I\cong\Z^4$ is torsionless, and $S_{a_0}(b_0)=b_0\neq \pm b_2=\pm S_{a_1}(b_0)$.
\end{proof}

\begin{corollary}
    There exists a canonical embedding
    \[
    \Aut(\Z[\R_8])\hookrightarrow\GL_4(\Z).
    \]
\end{corollary}

\begin{proof}
    Denote $I\coloneq I_\Z$. Recall that the only nonabelian nontrivial semidirect product of order $8$ is the dihedral group $D_4$. Therefore, Theorem \ref{thm:t2} provides isomorphisms $\Aut(\Z[\R_8]/I)\cong\Aut(\Z[\R_4])\cong D_4\cong\Aff_1(\Z_4)$, where we view $\Aff_1(\Z_4)$ as the group of invertible affine transformations of the affine line $\mathbb A^1_{\Z_4}=\Z_4$.  
    Under this identification, each $\overline\phi\in\Aff_1(\Z_4)$ acts on $\Z[\R_8]/I$ via $\pi(a_i)^{\overline\phi}\coloneq\pi(a_{\overline\phi(i\bmod4)})$. 
    Let \[\Phi\colon \Aut(\Z[\R_8])\hookrightarrow \Aut_{\mathsf{Ab}}(I)\times\Aff_1(\Z_4),\qquad \phi\mapsto(\phi|_I,\overline\phi)\] be the embedding constructed in the proof of Theorem \ref{prop:embedding}. It suffices to show that the projection
    \[
    \Pi\colon\Phi(\Aut(\Z[\R_8]))\to \Aut_{\mathsf{Ab}}(I),\qquad \Phi(\phi)\mapsto\phi|_I
    \]
    is injective, so fix $\Phi(\phi)\in\ker\Pi$. We have to show that $\overline\phi=\id_{\Z_4}$; since $\overline\phi\in\Aff_1(\Z_4)$, it suffices to show that $\overline{\phi}(0)=0$ and $\overline\phi(1)=1$. 

    Recall that each right multiplication map $S_{a_j}$ restricts to an abelian group automorphism of $I$. For every (set-theoretic) section $\pi\inv$ of $\pi$ and all $j\in\Z_8$, we have \[S_{\phi(a_j)}|_I=S_{\pi\inv(\pi(a_j)^{\overline\phi})}|_I=S_{a_{\overline\phi(j\bmod4)}}|_I\in\{S_{a_i}|_I\mid 0\leq i\leq 3\}=\{\pm S_{a_0}|_I,\pm S_{a_1}|_I\},\] 
    where in the first and last equalities we have used \eqref{eq:saj2} and \eqref{eq:saj}, respectively. It follows from Lemma \ref{lem:injective} that the assignment of $j\in\Z_4$ to $S_{a_j}|_I$ is injective; that is, the equation $S_{a_j}|_I=S_{a_i}|_I$ implies that $j=i$ in $\Z_4$. On the other hand, since $\phi$ is a ring endomorphism and $\phi|_I=\id|_I$, it is immediate that $S_{a_j}|_I=S_{\phi(a_j)}|_I=S_{a_{\overline\phi(j\bmod 4)}}$ for all $j\in\Z_8$. Therefore, taking $j=0,1$ completes the proof.
\end{proof}


\section{Commutator width in quandle algebras}\label{commutator-width}
Let $Q$ be a quandle, and let $\k$ be a commutative and associative ring with unity. Define the \textit{commutator} of elements $u, v \in \k[Q]$ to be the element $$[u, v]\coloneq uv-vu\in\k[Q].$$ Then the \textit{commutator subalgebra} $\k[Q]'$ is the $\k$-subalgebra of $\k[Q]$ generated by the set of all commutators in $\k[Q]$. For example, $Q$ is a commutative quandle if and only if $\k[Q]' =\{0\}$. Since $\varepsilon([u, v])=0$ for each commutator $[u, v] \in \k[Q]'$, we obtain the following:

\begin{lemma}\label{commutator-augmentation}
$\k[Q]' \le \Delta_R(Q)$.
\end{lemma}

The equality in the preceding lemma does not hold in general. For example, the dihedral quandle $\R_3$ is commutative, so $\mathbb{Z}[\R_3]'=0$. On the other hand, $\Delta(\R_3) =\langle e_1,  e_2 \rangle \neq 0$. 

\par

We define the \textit{commutator length} $\cl(u)$ of an element $u \in \k[Q]'$ as 
$$\cl(u)\coloneq\min \left\{n: u= \sum_{i=1}^n \alpha_i [u_i,v_i], \textrm{ where }\alpha_i \in R\text{ and }u_i, v_i \in \k[Q]\right\}.$$
The \textit{commutator width} $\cw(\k[Q])$ is defined as 
$$\cw(\k[Q])\coloneq\sup \{\cl(u) \mid u \in \k[Q]'\}.$$

 We remark that the analogous problem of computation of  commutator width of free Lie rings \cite{Bardakov}, free metabelian Lie algebras \cite{Poroshenko} and absolutely free and free solvable Lie rings of finite rank \cite{Roman'kov} has been considered in the literature.
\par

It follows from the definition of commutator width that a quandle $Q$ is commutative if and only if  $\cw(\k[Q])=0$. Consequently, we have $\cw(\k[\R_3])=0$.
\par

\begin{theorem}[{\cite[Thm.\ 7.4]{BPS-1}}] Let $\k$ be a unital associative commutative ring. Then:
\begin{enumerate}
\item If $\T$ is a trivial quandle, then $\cw(\k[\T])=1$.
\item $\cw(\k[\R_4]) =1$.
\item $\cw(\k[\Cs(4)])=1$.
\end{enumerate}
\end{theorem}

\begin{problem}[{\cite[Prob.\ 7.5]{BPS-1}}]\label{prob:cw}
Compute the commutator width of the quandle algebras of dihedral and free quandles.
\end{problem}

\subsection{Results}
Using the rack and quandle library in \cite{VY} and an exhaustive search in \texttt{GAP} \cite{GAP4}, we obtained the following results over finite fields $\F_q$. This partially answers Problem \ref{prob:cw}. The \texttt{GAP} code is available in a GitHub repository of the second author \cite{GitHub}.

\begin{theorem}\label{Thm:cw-R_n}
    \phantom{-}
    \begin{enumerate}
        \item Let $1\leq n\leq 21$ with $n\neq 3$. Then $\cw(\F_2[\R_n])=1$.
        \item Let $1\leq n\leq 14$ with $n\neq 3$, and let $p\in\{2,3\}$. Then $\cw(\F_p[\R_n])=1$.
        \item Let $1\leq n\leq 9$ with $n\neq 3$, and let $q\in\{2,3,4,5,7\}$. Then
        \[
        \cw(\F_q[\R_n])=\begin{cases}
            2 &\text{if }(n,q)=(5,5),\\
            1&\text{otherwise.}
        \end{cases}
        \]
    \end{enumerate}
\end{theorem}

\begin{theorem}\label{Thm: cw-Q}
    Let $Q$ be a noncommutative quandle of order at most $7$. Then for all $q\in\{2,3,4,5\}$, we have
    \[
    \cw(\F_q[Q])=\begin{cases}
        2&\text{if }Q\cong\R_5\text{ and }q=5,\\
        1&\text{otherwise.}
    \end{cases}
    \]
\end{theorem}

\begin{theorem}\label{thm:cw-8}
    Let $Q$ be a quandle of order $8$, and let $p\in\{2,3\}$. If $Q$ is isomorphic to the quandle \texttt{LRQ.Quandle(8, 1367)} in the library of \cite{VY}, then $\cw(\F_p[Q])=2$. Otherwise, $\cw(\F_p[Q])=1$.
\end{theorem}

We note that \texttt{LRQ.Quandle(8, 1367)} is the quandle with underlying set $\{b_1,\dots,b_8\}$ and right-multiplication maps
\begin{align*}
    S_{b_1}=S_{b_3}&\coloneq(2\,5\,7)(4\,6\,8),\\
    S_{b_2}=S_{b_8}&\coloneq(1\,7\,5)(3\,6\,4),\\
    S_{b_4}=S_{b_5}&\coloneq(1\,2\,7)(3\,6\,8),\\
    S_{b_6}=S_{b_7}&\coloneq(1\,5\,2)(3\,4\,8)
\end{align*}
in cycle notation.

As far as the authors are aware, the above results provide the first examples of quandle algebras whose commutator widths are greater than $1$. Explicitly, for $\R_5=\{a_0,\dots,a_4\}$, the \texttt{GAP} search found that $\dim_{\F_5}(\F_5[\R_5]')=4$, and exactly $40$ out of the $5^4=625$ vectors in $\F_5[\R_5]'$ have commutator length $2$; for example,
\[
\cl(2a_0+a_1+a_2+2a_3+4a_4)=2.
\]
For $Q=\texttt{LRQ.Quandle(8, 1367)}$, exactly $8$ of the $2^7=128$ vectors in $\F_2[Q]'$ have commutator length $2$; for example,
\[
\cl(b_3+b_5+b_6+b_8)=2.
\]
Exactly $324$ of the $3^7=2{,}187$ vectors in $\F_3[Q]'$ have commutator length $2$; for example,
\[
\cl(2b_3+2b_5+2b_6+2b_7+b_8)=2.
\]


\begin{ack}
  We would like to thank Valeriy Bardakov for helpful comments. The first author would also like to thank the participants of the \'Evariste Galois Seminar. The second author was supported by the K.\ Leroy Irvis Summer Research Fellowship at the University of Pittsburgh.
\end{ack}


\begin{thebibliography}{HD}
\bibitem{Andruskiewitsch}  N. Andruskiewitsch and M. Gra\~{n}a,  \textit{From racks to pointed Hopf algebras}, Adv. Math. {\bf 178} (2003), no. 2, 177--243.

\bibitem{Bardakov} V. G. Bardakov, \textit{Computation of commutator length in free groups}, Algebra Logika {\bf 39} (4) (2000),  395--440.

\bibitem{BDS} V. G. Bardakov, P. Dey, and M. Singh, \textit{Automorphism groups of quandles arising from groups},  Monatsh. Math. {\bf 184} (2017), 519--530.

\bibitem{BE} V. Bardakov and M. Elhamdadi, \textit{Idempotents and powers of ideals in quandle rings}, to appear in Arab. J. Math. (Springer).

\bibitem{BarTimSin} V. G. Bardakov, T. R. Nasybullov, and M. Singh, \textit{Automorphism groups of quandles and related groups},  Monatsh. Math.  {\bf 189} (2019), 1--21.

\bibitem{BPS} V. G. Bardakov, I. B. S. Passi, and M. Singh, \textit{Quandle rings},  J. Algebra Appl. {\bf 18} (2019), no. 8, 1950157, 23 pp.

\bibitem{BPS-1} V. G. Bardakov, I. B. S. Passi, and M. Singh, \textit{Zero-divisors and idempotents in quandle rings}, Osaka J. Math. {\bf 59} (2022),
611--637.

\bibitem{BSS1} V. G. Bardakov, Mahender Singh, and Manpreet Singh, \textit{Free quandles and knot quandles are residually finite}, Proc. Amer. Math. Soc. {\bf 147} (2019), no. 8, 3621--3633.

\bibitem{BSS2} V. G. Bardakov, Mahender Singh, and Manpreet Singh, \textit{Link quandles are residually finite},  Monatsh. Math. {\bf 191} (2020), 679--690.


\bibitem{BES} V. G. Bardakov, M. Elhamdadi, and M. Singh, \textit{Algebra of the Yang--Baxter equation: Skew braces, quandles, and cohomology}, to appear in Springer Monographs in Mathematics.
	
\bibitem{Carter} J. Scott Carter, \textit{A survey of quandle ideas}, Introductory lectures on knot theory, 22--53, Ser. Knots Everything, 46, World Sci. Publ., Hackensack, NJ (2012).

\bibitem{circulant} P.~J. Davis, {\it Circulant matrices}, A Wiley-Interscience Publication Pure and Applied Mathematics, John Wiley \& Sons, New York-Chichester-Brisbane, 1979; MR0543191

\bibitem{Eisermann} M. Eisermann, \textit{Yang--Baxter deformations of quandles and racks}, Algebr. Geom. Topol. {\bf 5} (2005), 537--562.

\bibitem{EFT} M. Elhamdadi, N. Fernando, and B. Tsvelikhovskiy, \textit{Ring theoretic aspects of quandles}, J. Algebra {\bf 526} (2019), 166--187.

\bibitem{Elhamdadi2012} M. Elhamdadi, J. Macquarrie,  and R. Restrepo,  \textit{Automorphism groups of quandles}, J. Algebra Appl. {\bf 11} (2012), 1250008, 9 pp.

\bibitem{MBM} M. Elhamdadi, B. Nunez, and M. Singh, \textit{Enhancements of link colorings via idempotents of quandle rings}, J. Pure Appl. Algebra {\bf227} (2023), no. 10, Paper No. 107400.

\bibitem{MBMD} M. Elhamdadi, B. Nunez, M. Singh, and D. Swain, \emph{Idempotents, free products and quandle coverings}, Internat. J. Math. 34 (2023), no. 3, Paper No. 2350011, 27 pp.

\bibitem{ETV} M. Elhamdadi, L. Ta, and B. Virgin, \emph{Idempotents in quandle algebras}, in preparation. 

\bibitem{GAP4} The GAP Group, GAP -- Groups, Algorithms, and Programming, Version 4.16.0; 2026. \url{https://www.gap-system.org}

\bibitem{Joyce-Thesis} D. Joyce, \textit{An algebraic approach to symmetry with applications to knot theory}, PhD Thesis, University of Pennsylvania, 1979. vi+63 pp.

\bibitem{Joyce} D. Joyce, \textit{A classifying invariant of knots, the knot quandle}, J. Pure Appl. Algebra {\bf 23} (1982), 37--65.

\bibitem{Kamada} S. Kamada, \textit{Knot invariants derived from quandles and racks}, Invariants of knots and 3-manifolds (Kyoto, 2001), 103--117 (electronic), Geom. Topol. Monogr., 4, Geom. Topol. Publ., Coventry, (2002).

\bibitem{Loos}  O. Loos,  \textit{Reflection spaces and homogeneous symmetric spaces}, Bull. Amer. Math. Soc. {\bf 73} (1967) 250--253.

\bibitem{Matveev} S. V. Matveev, \textit{Distributive groupoids in knot theory}, (Russian) Mat. Sb. (N.S.) {\bf119(161)} (1982), no. 1, 78--88, 160.

\bibitem{NP} P. Narayanan and S. Panja, \textit{Idempotents in integral ring of dihedral quandle}, 2022. Preprint, arXiv: 2206.10386.

\bibitem{Nelson} S. Nelson, \textit{The combinatorial revolution in knot theory}, Notices Amer. Math. Soc. {\bf 58} (2011), 1553--1561.

\bibitem{poole} D.~G. Poole, \textit{The stochastic group}, Amer. Math. Monthly {\bf 102} (1995), no.~9, 798--801; MR1357724

\bibitem{Poroshenko} E. N. Poroshenko, \textit{Commutator width of elements in a free metabelian Lie algebra}, Algebra and Logic {\bf 53} (2014), 377--396.

\bibitem{Roman'kov} V. A. Roman'kov: \textit{The commutator width of some relatively free Lie algebras and nilpotent groups},  Sib. Math. J. {\bf 57} (2016), no. 4, 679--695

\bibitem{GitHub} L. Ta, \textit{Commutator width of quandle algebras}, GitHub, \url{https://github.com/luc-ta/commutator-width-of-quandle-algebras/tree/main}

\bibitem{LT2} L. Ta, \textit{From affine algebraic racks to Leibniz algebras and Yang--Baxter operators}, J. Algebra \textbf{712} (2027), 55--100; MR5105987

\bibitem{LT1} L. Ta, \textit{On medial Latin quandles and affine modules}, to appear in Canad. Math Bull., doi:10.4153/S0008439526102306

\bibitem{LT} L. Ta, \textit{Structure theory of commutative quandles and medial Latin quandles}. Blog post. \url{https://luc-ta.github.io/blog/2026/commutative-quandles/}

\bibitem{VY} P. Vojt\v echovsk\'y{} and S.~Y. Yang, \emph{Enumeration of racks and quandles up to isomorphism}, Math. Comp. \textbf{88} (2019), no. 319, 2523--2540.

\end{thebibliography}
\end{document}